\documentclass[preprint,12pt]{elsarticle}
\usepackage[utf8]{inputenc}

\usepackage{fullpage}

\usepackage{xcolor}
\usepackage{amsmath}
\usepackage{amssymb}
\usepackage{latexsym}
\usepackage{amsthm}

\usepackage{tikz}

\usepackage{url}

\newtheorem{theorem}{Theorem}
\newtheorem{lemma}[theorem]{Lemma}

\newtheorem{proposition}[theorem]{Proposition}

\newtheorem{remark}[theorem]{Remark}

\usepackage{minitoc}

\usepackage{ifpdf}

\journal{Discrete Applied Mathematics} 

\begin{document}

\begin{frontmatter}



\title{{Minimal obstructions to $(s,1)$-polarity in cographs}}


\author[FC]{Fernando Esteban Contreras-Mendoza}
\ead{esteban.contreras.math@ciencias.unam.mx}

\author[CINVESTAV]{C\'esar Hern\'andez-Cruz\corref{cor1}}
\ead{cesar@cs.cinvestav.mx}

\address[FC]{Facultad de Ciencias\\
Universidad Nacional Aut\'onoma de M\'exico\\
Av. Universidad 3000, Circuito Exterior S/N\\
C.P. 04510, Ciudad Universitaria, D.F., M\'exico}

\address[CINVESTAV]{Departamento de Computaci\'on\\
Centro de Investigaci\'on y de Estudios Avanzados del IPN\\
Av. Instituto Polit\'ecnico Nacional 2508\\
C.P. 07300, D.F., M\'exico
}

\cortext[cor1]{Corresponding author}

\begin{abstract}
Let $k,l$ be non negative integers.   A graph $G$ is $(k,l)$-polar
if its vertex set admits a partition $(A,B)$ such that $A$
induces a complete multipartite graph with at most $k$ parts,
and $B$ induces a disjoint union of at most $l$ cliques
with no other edges.   A graph is a cograph if it does not
contain $P_4$ as an induced subgraph.

It is known that $(k,l)$-polar cographs can be characterized
through a finite family of forbidden induced subgraphs, for any
fixed choice of $k$ and $l$.   The problem of determining the
exact members of such family for $k = 2 = l$ was posted by
Ekim, Mahadev and de Werra, and recently solved by Hell,
Linhares-Sales and the second author of this paper.   So far,
complete lists of such forbidden induced subgraphs are known
for $0 \le k,l \le 2$; notice that, in particular, $(1,1)$-polar
graphs are precisely split graphs.   

In this paper we focus on this problem for $(s,1)$-polar
cographs.   As our main result, we provide a recursive
complete characterization of the forbidden induced
subgraphs for $(s,1)$-polar cographs, for every non
negative integer $s$. Adittionaly, we show that cographs
having an $(s,1)$-partition for some integer $s$ (here $s$
is not fixed) can be characterized by forbidding a family of
four graphs.
\end{abstract}


\end{frontmatter}



\section{Introduction}
\label{sec:Introduction}

All graphs in this paper are considered to be finite and simple.  
We refer the reader to \cite{bondy2008} for basic terminology and
notation.  In particular, we use $P_k$ and $C_k$ to denote the path
and cycle on $k$ vertices, respectively.

Cographs were introduced by Corneil, Lerchs and Stewart
Burlingham in \cite{corneilDAM3}.   A graph is a complement
reducible graph, or {\em cograph}, if it can be constructed
using the following rules.
\begin{itemize}
	\item $K_1$ is a cograph.
	
	\item If $G$ is a cograph, then its complement $\overline{G}$
		is also a cograph.
		
	\item If $G$ and $H$ are cographs, then the disjoint
		union $G + H$ is also a cograph.
\end{itemize}

In \cite{corneilDAM3}, seven characterizations of this family
were presented; in this work we will extensively use two very
well known of these.   A graph is a cograph if and only if it is
$P_4$-free (it does not contain $P_4$ as an induced
subgraph), if and only if the complement of any of its
nontrivial connected subgraphs is disconnected.

In 1990, Peter Damaschke proved that the class of
cographs is well quasi-ordered by the induced subgraph
relation \cite{damaschkeJGT14}; in other words, every
hereditary property of graphs can be characterized by
a finite family of forbidden induced subgraphs.   Thus,
finding the family of minimal forbidden induced
subgraphs characterizing a given hereditary property
in the class of cographs comes as a natural problem.
The knowledge of such families has two obvious
consequences, first, analyzing the structure of
the members of one of this families (for example,
fixing a certain value of a parameter which the
hereditary property depends on) may shed some light
on the general problem.   Also, the members of these
families are no-certificates for the associated decision
problem.   Consider for example a generalized colouring
problem (partition the set of vertices of a graph into $k$
subsets such that each part has some hereditary
property), if we know the complete list of minimal
forbidden induced subgraphs, an algorithm could be
designed to receive a cograph $G$, decide if it has a
generalized colouring of the desired type, and return
either the colouring of $G$ (a yes-certificate) or one
of the forbidden induced subgraphs (a no-certificate).
Such an algorithm is known as a {\em certifying
algorithm}, and if the validity of the certificates can be
verified efficiently (faster than the original algorithm),
having a certifying algorithm makes it possible to verify
the correctness of its implementations.

In the present work, we will focus on polar partitions.
A {\em polar partition} of a graph $G$ is a partition of
the vertices of $G$ into parts $A$ and $B$ in such a
way that  the subgraph induced by $A$ is a complete
multipartite graph and the subgraph induced by $B$
is the complement of a complete multipartite graph.  A
graph $G$ is {\em polar} if it admits a polar partition,
and is $(s,k)${\em -polar} if it admits a polar partition
$(A,B)$ in which $A$ has at most $s$ parts and
$\overline{B}$ at most $k$ parts. When $s = 1$, an
$(s,k)$-polar graph (partition) is called a {\em monopolar}
graph (partition).   Clearly, for any fixed non negative
integers $s$ and $k$, having an $(s,k)$-partition is a
hereditary property, and thus, as we have already
mentioned, $(s,k)$-polar cographs can be characterized
by a finite familiy of forbidden induced subgraphs. A {\em
cograph minimal $(s,k)$-polar obstruction} is a cograph
which is not $(s,k)$-polar, but such that every proper
induced subgraph is. A {\em cograph $(s,k)$-polar
obstruction} is simply a cograph which is not $(s,k)$-polar.

Polar graphs have received considerable attention in the
literature since Chernyak and Chernyak proved in
\cite{chernyakDM62} that their recognition problem is
$\mathcal{NP}$-complete.   Surprisingly, Farrugia proved
in \cite{farrugiaEJC11} that the problem remains
$\mathcal{NP}$-complete even for monopolar graphs,
and Churchley and Huang proved in \cite{churchleyJGT76},
that monopolar recognition remains
$\mathcal{NP}$-complete even when restricted to
triangle-free graphs.   Regarding these two problems,
the class of claw-free graphs is interesting, it distinguishes
monopolarity, which is polynomial time recognizable, from
polarity, which is $\mathcal{NP}$-complete,
\cite{churchleyJGT76}.

We think that it is worth noticing that polar
partitions are a particular case of a more general
kinf of partition problems, namely, {\em matrix
partitions}.  The concept of a matrix partition
unifies many interesting graph partition problems.
Given a symmetric $n \times n$ matrix $M$,
with entries in $\{ 0, 1, \ast \}$, an $M$-{\em
partition} of a graph $G$ is a partition\footnote{As
it is usual in graph theory, we do not require
every part of the partition to be non-empty.}
$(V_1, \dots, V_n)$ of $V(G)$ such that, for
every $i, j \in \{ 1, \dots, n \}$,
\begin{itemize}
	\item $V_i$ is completely adjacent to
		$V_j$ if $M_{ij} = 1$,
	
	\item $V_i$ is completely non-adjacent
		to $V_j$ if $M_{ij} = 0$,
	
	\item There are no restrictions if $M_{ij}
		= \ast$.
\end{itemize}
It follows from the definition that, in particular, if
$M_{ii} = 0$ ($M_{ii} = 1$), then $V_i$ is a stable
set ($V_i$ is a clique).   The $M$-{\em partition
problem} asks whether or not an input graph $G$
admits an $M$-partition. See \cite{survey} for
a survey on the subject. It is easy to see that
an $(s,k)$-polar partition of $G$ is a matrix partition
in which the matrix $M$ has $s+k$ rows and
columns, the principal submatrix induced by the
first $s$ rows is obtained from an identity matrix by
exchanging $0$'s and $1$'s, the principal submatrix
induced by the last $k$ rows is an identity matrix,
and all other entries are $\ast$.   Therefore, it
follows from \cite{federSIAMJDM16}, that for any
fixed $s$ and $k$, the class of $(s,k)$-polar
graphs can be recognized in polynomial time.
Feder, Hell and Hochst\"attler proved in
\cite{feder2006} that if $M$ is a matrix where all
the off-diagonal entries of the principal submatrix
with zeroes on the diagonal are equal to $a$, all
the off-diagonal entries of the principal submatrix
with only ones on the diagonal are equal to $b$
and all the remaining entries of $M$ are equal to
$c$, with $a, b, c \in \{ 0, 1, \ast \}$, then every
cograph minimal $M$-obstruction has at most
$(k+1)(\ell+1)$ vertices.

For very small values of $s$ and $k$ the minimal
$(s,k)$-polar obstructions are well known; a graph
is $(0,k)$-polar if and only if it is a dijoint union of
at most $k$-cliques, it is $(s,0)$-polar if and only
if it is a complete $s$-partite graph, and it is
$(1,1)$-polar if and only if it is a split graph. It was
shown by Foldes and Hammer \cite{foldesSECGTC}
that a graph is split if and only if it is $\{ 2K_2, C_4,
C_5 \}$-free; it is folklore that a graph is a disjoint
union of at most $k$-cliques if and only if its
independence number is at most $k$ and it is
$P_3$-free, which by complementation implies that
a graph is a complete $s$-partite graph if and only
if it is $\{ K_{s+1}, K_1 + K_2 \}$-free.

For cographs,  Ekim, Mahadev and de Werra
proved in \cite{ekimDAM156} that there are only
eight cograph minimal polar obstructions, and
sixteen cograph minimal $(s,k)$-polar obstructions
when $\min\{ s, k \} = 1$, \cite{ekimDAM171}. In the
same paper, they proposed the problem of finding a
characterization of $(2,2)$-polar cographs; this
problem was solved by Hell, Hern\'andez-Cruz
and Linhares-Sales in \cite{hellDAM}, where they
proved that there are $48$ cograph minimal
$(2,2)$-polar obstructions.   The exhaustive list
of nine cograph minimal $(2,1)$-polar obstructions
was found by Bravo, Nogueira, Protti and Vianna,
\cite{bravo}.

In this work, we show that there are precisely four
cograph minimal monopolar obstructions (see Figure
\ref{fig:essentials}), and provide a recursive
a recursive characterization for cograph minimal
$(s,1)$-polar obstructions.   By taking complements it
trivial to obtain analogous results for $(1,k)$-polar
cographs.

We will denote the complement of $G$ by $\overline{G}$.
We say that a component of $H$ is {\em trivial} or an
{\em isolated vertex} if it is isomorphic to $K_1$. A
{\em $k$-cluster} is the complement of a complete
$k$-partite graph, i.e., a disjoint union of $k$-cliques
without any other edges.

Given graphs $G$ and $H$, the disjoint union of $G$
and $H$ is denoted by $G + H$, and the join of $G$
and $H$ is denoted by $G \oplus H$.   Thus, the sum
of $n$ disjoint copies of $G$ is denoted by $nG$, and
for disjoint graphs $G_1, \dots, G_k$, their disjoint
union is denoted as $\sum_{i=1}^k G_k$.

The rest of the paper is organized as follows.   In
Section \ref{sec:PreliminaryResults}, we prove some
technical lemmas that will be used in Section
\ref{sec:MainResults} to prove our main results.
In Section \ref{sec:Number}, a brief asymptotic estimation
of the number of cograph minimal $(s,1)$-obstructions
is given.   Conclusions and future lines of work are
presented in Section \ref{sec:Conclusions}.

\section{Preliminary results}\label{sec:PreliminaryResults}

We begin this section by characterizing graphs that are
cograph minimal $(s,1)$-obstructions for every integer
$s$, with $s \ge 2$.   We will call such obstructions {\em
essential}.   First, notice that if $G$ is an $(s,1)$-polar graph
with polar partition $(A,B)$, then $B$ is just a clique, and
$G[V - B] = G[A]$, is a complete multipartite graph.   On the
other hand, if $G$ is a graph containing a clique $K$ such
that $G[V - K]$ is a complete multipartite graph, then clearly
$(V-K, K)$ is an $(s,1)$-polar partition of $G$.   This simple
observation is contained in the following remark.

\begin{remark}\label{rem:essential}
Let $G$ be a cograph. If for every clique $K$ of $G$, the
induced subgraph $G[V - K]$ contains $\overline{P_3}$ as
an induced subgraph, then $G$ is not an $(s,1)$-polar
cograph for any integer $s$, $s \ge 2$.
\end{remark}

Now, we can show the existence of some essential
cograph essential $(s,1)$-polar obstructions.

\begin{lemma}\label{lem:EssentialObstructions}
The graphs $K_1 + 2K_2, \overline{K_2} + C_4, 2P_3$
and $K_1 + (\overline{P_3} \oplus \overline{K_2})$
depicted in Figure \ref{fig:essentials} are cograph minimal
$(s,1)$-polar obstructions for every integer $s$, $s\ge 2$.
\end{lemma}

\begin{figure}[h!]
\begin{center}
\begin{tikzpicture}
[every circle node/.style ={circle,draw,minimum size= 5pt,
inner sep=0pt, outer sep=0pt},
every rectangle node/.style ={}];

\begin{scope}[scale=1.25]
\node [circle] (1) at (0,0)[]{};
\node [circle] (2) at (1,0)[]{};
\node [circle] (3) at (0,0.8)[]{};
\node [circle] (4) at (1,0.8)[]{};
\node [circle] (5) at (0.5,1.5)[]{};
\foreach \from/\to in {1/2,3/4}
\draw [-, shorten <=1pt, shorten >=1pt, >=stealth, line width=.7pt] 
(\from) to (\to);
\node [rectangle] (1) at (0.5,-0.75){$K_1+2K_2$};
\end{scope}

\begin{scope}[xshift=90, scale=1.2]
\node [circle] (1) at (0,0)[]{};
\node [circle] (2) at (1,0)[]{};
\node [circle] (3) at (0,1)[]{};
\node [circle] (4) at (1,1)[]{};
\node [circle] (5) at (0,1.8)[]{};
\node [circle] (6) at (1,1.8)[]{};
\foreach \from/\to in {1/2,1/3,3/4,2/4}
\draw [-, shorten <=1pt, shorten >=1pt, >=stealth, line width=.7pt] 
(\from) to (\to);
\node [rectangle] (1) at (0.5,-0.75){$\overline{K_2}+C_4$};
\end{scope}

\begin{scope}[xshift=180, scale=1.2]
\node [circle] (1) at (0,0)[]{};
\node [circle] (2) at (0,0.9)[]{};
\node [circle] (3) at (0,1.8)[]{};
\node [circle] (4) at (1,0)[]{};
\node [circle] (5) at (1,0.9)[]{};
\node [circle] (6) at (1,1.8)[]{};
\foreach \from/\to in {1/2,2/3,4/5,5/6}
\draw [-, shorten <=1pt, shorten >=1pt, >=stealth, line width=.7pt] 
(\from) to (\to);
\node [rectangle] (1) at (0.5,-0.75){$2P_3$};
\end{scope}

\begin{scope}[xshift=270, scale=1.25]
\node [circle] (1) at (0.5,0)[]{};
\node [circle] (2) at (-0.1,0.6)[]{};
\node [circle] (3) at (0.5,1.2)[]{};
\node [circle] (4) at (1.1,0.6)[]{};
\node [circle] (5) at (0.5,0.6)[]{};
\node [circle] (6) at (0.5,1.8)[]{};
\foreach \from/\to in {1/2,2/3,3/4,4/1,1/5,2/5,4/5}
\draw [-, shorten <=1pt, shorten >=1pt, >=stealth, line width=.7pt] 
(\from) to (\to);
\node [rectangle] (1) at (0.5,-0.75){$K_1+(\overline{P_3}\oplus
\overline{K_2})$};
\end{scope}

\end{tikzpicture}
\end{center}
\caption{Essential obstructions.}
\label{fig:essentials}
\end{figure}
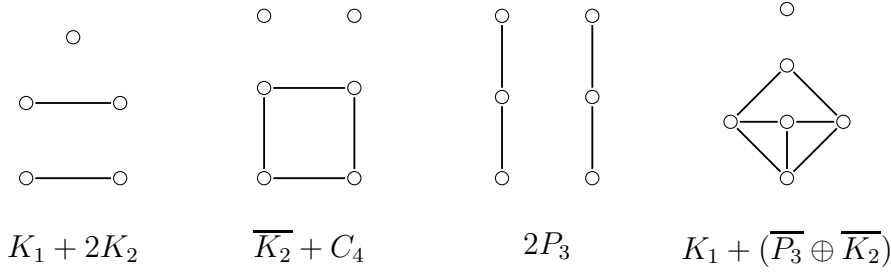

	\begin{proof}
	It is evident that all the graphs shown in Figure
	\ref{fig:essentials} are cographs.   By a simple exploration
	taking into consideration Remark \ref{rem:essential}, it is
	routine to verify that none of these graphs is an $(s,1)$-polar
	cograph for any positive integer $s$. Furthermore it is easy
	to verify that in each of these graphs the deletion of any
	vertex results in a $(2,1)$-polar cograph, so all of them are
	cograph minimal $(s,1)$-polar obstructions for any integer
	$s$ greater than or equal to 2.
	\end{proof}

Notice that all essential obstructions are disconnected,
and, since they are small graphs, it is not hard to imagine
that they will prevent larger disconnected minimal
obstructions to exist.  Our next lemma concretes this
intuitive idea, showing that disconnected cograph minimal
$(s,1)$-polar obstructions have at most two components,
except for $K_1 + 2K_2$ and $\overline{K_2} + C_4$;
some additional restrictions on the structure of such
minimal obstructions are also obtained.

\begin{lemma}\label{lem:TwoComponents}
Let $s$ be an integer, $s\geq 2$. Then every cograph
minimal $(s, 1)$-polar obstruction different from $K_1
+ 2K_2$ and $\overline{K_2} + C_4$ has at most
two connected components.

Moreover, if a cograph
minimal $(s,1)$-polar obstruction has two connected
components and it is neither $2P_3$ nor $2K_{s+1}$,
then one of its components is $K_1$ or $K_2$, and its
other component is not a complete graph.
\end{lemma}

	\begin{proof}
	Let $G$ be a cograph minimal $(s, 1)$-polar obstruction
	with at least three connected components. Observe
	that since $G$ is not a split graph, $G$ contains $C_4$
	or $2K_2$ as an induced subgraph.   In the former case,
	since $G$ has at least three connected components, $G$
	contains $\overline{K_2} + C_4$ as an induced subgraph. For
	the latter case, again, noting that $G$ has at least
	three components leads to conclude that $G$ contains
	$K_1 + 2K_2$ as an induced subgraph.   By the
	minimality of $G$, the previous observations imply
	that $G$ is isomorphic to $\overline{K_2} + C_4$ or
	$K_1 + 2K_2$. So we have that every cograph minimal
	$(s,1)$-polar obstruction isomorphic to neither
	$K_1 + 2K_2$ nor $\overline{K_2} + C_4$ has at most
	two connected components.
	
	Now, suppose that $G$ is a cograph minimal
	$(s, 1)$-polar obstruction isomorphic to neither
	$2P_3$ nor $2K_{s+1}$, and with two connected
	components. Note that since $G$ is $2P_3$-free, at
	least one of the components of
	$G$ is a complete graph.  If both components of $G$
	are complete graphs, then both of them must have at
	least $s+1$ vertices, otherwise $G$ would be $(s,1)$-polar;
	but in this case $G$ should be isomorphic to $2K_{s+1}$.
	Thus, we may assume that one component of $G$ is a
	complete graph and the other one is not.
	
	Finally, suppose for a contradiction that the
	complete component of $G$ has three or more vertices,
	and let $v$ be one of these vertices.   By the
	minimality of $G$ we have that $G-v$ admits an
	$(s,1)$-polar partition $(A,B)$.   If $B$ is contained
	in the complete component of $G-v$, then $(A, B\cup \{v
	\})$ is an $(s,1)$-polar partition of $G$, a contradiction.
	Hence, $B$ is contained in the non-complete component
	of $G$.   Clearly, $(G-v)-B$ contains $\overline{P_3}$ as
	an induced subgraph, and thus, it cannot be covered
	by $A$, contradicting the choice of $(A,B)$ as an
	$(s,1)$-polar partition of $G-v$.   Since the contradiction
	arises from assuming that the complete component of $G$
	has at least three vertices, then it should have at
	most two vertices.
	\end{proof}

So, it follows from the previous lemma that we can assume
that every disconnected cograph minimal $(s,1)$-obstruction
contains either an isolated vertex or a component isomorphic
to $K_2$.  The following two lemmas describe the structure of the
cograph minimal $(s,1)$-obstructions with two components,
other than the essential obstructions and $2K_{s+1}$.   It is
a bit surprising that for any integer $s$ greater than or equal
to $2$, there are only two such obstructions.

\begin{lemma}\label{lem:H+K2}
Let $s$ be an integer, $s\geq 2$, and let $H$ be a
connected cograph such that $G=H+K_2$ is a cograph
minimal $(s,1)$-polar obstruction. Then $G$ is isomorphic to
$K_2 + (\overline{K_2}\oplus K_s)$.
\end{lemma}

	\begin{proof}
	Since $G$ is a cograph $(1,s)$-polar obstruction,
	$H$ is not a complete $s$-partite graph, so $H$
	has $K_{s+1}$ or $\overline{P_3}$ as induced
	subgraph. Nevertheless, if $H$ has $\overline{P_3}$
	as an induced subgraph, then $K_1+2K_2$
	is an induced subgraph of $G$, contradicting the
	minimality of $G$. Therefore $H$ has $K_{s+1}$ as
	induced subgraph.
	
	Let $K$ be a subset of $V(H)$ such that
	$H[K]\cong K_{s+1}$. Since $H$ is $\overline{P_3}$-free,
	each vertex of $H$ that is not in $K$ is adjacent
	to every vertex in $K$, except maybe to one of them.
	Moreover, since $G$ is a $(s,1)$-polar obstruction,
	 $H$  is not a complete graph, and in consequence
	there is a vertex $v$ of $H$ that is non-adjacent
	to at least one vertex in $K$. Note that
	$H[K\cup\{v\}]$ is isomorphic to
	$\overline{K_2} \oplus K_s$, and hence,
	$G$ has $K_2 + (\overline{K_2} \oplus K_s)$
	as an induced subgraph. But it is easy to verify that
	$K_2 + (\overline{K_2} \oplus K_s)$ is a cograph
	minimal $(1,s)$-polar obstruction,
	so, from the minimality of $G$ we have that $G$ is
	isomorphic to $K_2 + (\overline{K_2} \oplus K_s)$.
	\end{proof}

\begin{lemma}\label{lem:H+K1}
Let $s$ be an integer, $s\geq 2$, and let $H$ be a
connected cograph such that $G=H+K_1$ is a cograph
minimal $(s,1)$-polar obstruction non isomorphic to
$K_1 + (\overline{K_2} \oplus \overline{P_3})$.
Then $G$ is isomorphic to $K_1 + (C_4\oplus K_{s-1})$.
\end{lemma}

	\begin{proof}
	Since $G$ is not a split graph, $G$ have $2K_2$ or $C_4$
	as an induced subgraph, and evidently these subgraphs
	must be induced subgraphs of $H$. Nevertheless, if $H$ have
	$2K_2$ as induced subgraph, then $G$ contains $K_1 +
	2K_2$ as an induced subgraph, and by the minimality of $G$,
	it must be isomorphic to $K_1 + 2K_2$, contradicting that $G$
	has only two connected components. So there is a subset $C$
	of the vertex set of $H$ that induces a $C_4$.
	
	Let $v$ a vertex of $H$ that is not in $C$, which must exist,
	or else $G$ would be $(2,1)$-polar. Then $v$ must
	be adjacent to some vertex of $C$, otherwise $G$ would have
	$\overline{K_2} + C_4$ as induced subgraph,
	which is not possible. On the other hand, since $G$
	is a cograph, $v$ cannot be adjacent to exactly one
	vertex of $C$ nor can be adjacent to exactly two
	adjacent vertices of $C$. Furthermore,
	if $v$ is adjacent to three vertices of $C$, then $C\cup\{v\}$
	induces $\overline{K_2} \oplus \overline{P_3}$, and therefore
	$G$ has
	$K_2 + (\overline{K_2} \oplus \overline{P_3})$ as an induced
	subgraph, which by the minimality of $G$ implies that $G$
	is isomorphic to $K_2 + (\overline{K_2} \oplus
	\overline{P_3})$, but we are assuming that $G$ is not.
	So we have that every vertex of $H$ that is not in $C$
	must be adjacent to every vertex of $C$, or must be
	adjacent to exactly a pair of non adjacent vertices of $C$.
	
	Let $D$ be the graph induced by the subset of vertices of $H$
	that are not in $C$ but such that are adjacent to every vertex
	in $C$. Notice that if $D$ were a complete $(s-2)$-partite
	graph, then $H$ would be a complete $s$-partite graph,
	and therefore $G$  would be a $(s,1)$-polar graph. Thus,
	since we are assuming that $G$ is a
	$(s,1)$-polar obstruction, $D$ cannot be a complete
	$(s-2)$-partite graph, and in consequence $D$ has
	$\overline{P_3}$ or $K_{s-1}$ as an induced subgraph.
	
	Nevertheless, we claim that $D$ is a
	$\overline{P_3}$-free graph.   Otherwise, if $D$
	has $\overline{P_3}$ as an induced subgraph, then,
	together with any two non adjacent vertex of $C$
	this would induce a $\overline{K_2} \oplus
	\overline{P_3}$, which cannot occur. Then
	$D$ has $K_{s-1}$ as an induced subgraph, and hence $G$
	have $K_1 + (C_4 \oplus K_{s-1})$ as induced subgraph. But
	$K_1 + (C_4 \oplus K_{s-1})$ is a cograph minimal
	$(s,1)$-polar obstruction, so $G$ is
	isomorphic to $K_1 + (C_4 \oplus K_{s-1})$.
	\end{proof}
	
	So far, we have characterized all disconnected cograph
	minimal $(s,1)$-polar obstructions, which are a constant
	number for any choice of $s$.   Taking into account that
	the number of minimal $(s,0)$-polar obstructions is two,
	regardless of the choice of $s$, it would seem possible
	to have a constant number of cograph minimal
	$(s,1)$-polar obstructions, we would only need to show
	that the number of such connected obstructions is a
	constant independent of $s$.   Unfortunately, this will
	not be the case.   It is easy to verify that a cograph $G$
	is a minimal $(s,1)$-polar obstruction if and only if
	$\overline{G}$ is a minimal $(1,s)$-polar obstruction.
	Thus, in order to characterize the connected cograph
	minimal $(s,1)$-polar obstructions, we will study
	their complements, the disconnected cograph minimal
	$(1,k)$-polar obstructions. 

\begin{lemma}\label{theo:Components Of(1,k)Obstructions}
Let $k$ be a nonnegative integer, and let $G$ be a cograph
minimal $(1,k)$-polar obstruction.
Then every component of $G$ is nontrivial, and if
$G$ is not isomorphic to $(k+1)K_2$ then $G$
has at most $k$ components.
\end{lemma}

	\begin{proof}
	Suppose for a contradiction that $G$ has an isolated
	vertex $v$. Since $G$ is a cograph minimal
	$(1,k)$-polar obstruction, $G-v$ admits a $(1,k)$-polar
	partition $(A,B)$, but in such case $(A\cup\{v\},B)$ is
	a $(1,k)$-polar partition of $G$, contradicting the
	minimality of $G$. Thus, we conclude that
	every component of $G$ has at least two vertices.
	
	On the other hand, $H=(k+1)K_2$ is a cograph
	$(1,k)$-polar obstruction, because every $k$-cluster $K$
	of $H$ intersect at most $k$ components of $H$,
	and therefore $H-K$ is a nonempty graph.
	Furthermore for every vertex $v\in V(H)$, $H-v$ is
	isomorphic to $kK_2+K_1$, which is clearly a
	$(1,k)$-polar cograph, so $H$ is a cograph minimal
	$(1,k)$-polar obstruction.
	
	Finally, if $G$ has more than $k$ components, since
	none of them is an isolated vertex, $G$ has $(k+1)K_2$
	as an induced subgraph, so that $G\cong (k+1)K_2$.
	Thus, if $G\not\cong(k+1)K_2$, then $G$ has
	at most $k$ components.
	\end{proof}

\section{Main results}\label{sec:MainResults}

In this section we will obtain a recursive characterization
of disconnected cograph minimal $(1,k)$-polar obstructions
to achieve our goal of characterizing all cograph minimal
$(s,1)$-polar obstructions.   We begin by describing a
construction of a cograph minimal $(1,k)$-polar obstruction
as a disjoint union of smaller minimal polar obstructions.

\begin{lemma}\label{teo:SumsOfObstructions}
Let $t$ be an integer, $t\geq 2$, and for each
$i\in\{1,\ldots, t\}$, let $G_i$ be a connected cograph
minimal $(1,k_i)$-polar obstruction that is a
$(1,k_i+1)$-polar graph. Then, for
$m = t-1 + \sum_{i=1}^t k_i$, the graph $G = G_1 +
\ldots + G_t$ is a cograph minimal $(1,m)$-polar
obstruction that is a $(1,m+1)$-polar graph.
\end{lemma}

	\begin{proof}
	Let $G_1,\ldots,G_t$ and $G$ be as in the hypothesis.
	We first prove by means of a contradiction that
	$G$ is a cograph $(1,m)$-polar obstruction.
	Suppose that $G$ admits a $(1,m)$-polar partition
	$(A,B)$, and define for each $i\in\{1,2,\ldots,t\}$
	the sets $A_i=V(G_i)\cap A$ and $B_i=V(G_i)\cap B$.
	Note that every component of $G[B]$ is contained
	in a component of $G$. Denote the number
	of components of $G_i[B_i]$ by $l_i$; if $k_i < l_i$
	for every $i \in \{ 1, \dots, t \}$, we would have
	$m+1 = \sum_{i=1}^t(k_i+1) \leq \sum_{i=1}^t l_i = m$,
	a contradiction.  Hence, there is $j \in \{1,\ldots,
	t\}$ such that $l_j\leq k_j$.   Nevertheless,
	we have that $G_i$ is a cograph $(1,k_i)$-polar
	obstruction and $(A_i,B_i)$ is a $(1,l_i)$-polar
	partition of $G_i$, so that $k_i<l_i$
	for every $i\in\{1,\ldots,t\}$, contradicting our
	previous argument.   Since the contradiction arises
	from assuming that $G$ is a $(1,m)$-polar cograph,
	we conclude that $G$ is a cograph $(1,m)$-polar
	obstruction.
	
	Now we prove that $G$ is minimal. If $v\in V(G)$,
	then $v\in V(G_j)$ for some $j\in\{1, \ldots,t\}$,
	say, without loss of generality,
	for $j=1$. Since $G_1$ is a cograph minimal
	$(1,k_1)$-polar obstruction, the graph $G_1-v$
	admits a $(1,k_1)$-polar partition $(A_1,B_1)$,
	and since by hypothesis $G_i$ is a $(1,k_i+1)$-polar cograph
	for each $i\in\{2,3,\ldots,t\}$, we have that $G_i$
	admits a $(1,k_i+1)$-polar partition $(A_i,B_i)$.
	Therefore, $G-v$ is a $(1,m)$-polar cograph
	with polar partition
	$(\bigcup_{i=1}^t A_i, \bigcup_{i=1}^t B_i)$.
	Thus, $G$ is a cograph minimal $(1,m)$-polar obstruction.
	
	Finally, since for each $i\in\{1,2,\ldots,t\}$ the graph
	$G_i$ admits a $(1,k_i+1)$-polar partition $(A_i,B_i)$,
	then $(\bigcup_{i=1}^t A_i, \bigcup_{i=1}^t B_i)$
	is a $(1,m+1)$-polar partition of $G$, and therefore
	$G$ is a $(1,m+1)$-polar cograph.
	\end{proof}
	
Our goal is to prove that the cographs described in Lema
\ref{teo:SumsOfObstructions} are the only disconnected cograph
minimal $(1,k)$-polar obstructions.   In order to achieve this we
need the following technical, yet simple, result.

\begin{lemma}\label{cor:SumsOfObstructions}
Let $t$ be an integer, $t\geq 2$, and for each
$i\in\{1,\ldots, t\}$, let $G_i$ be a connected cograph
minimal $(1,k_i)$-polar obstruction that is a
$(1,k_i+1)$-polar graph.   Then, for
$m = t-1 + \sum_{i=1}^t k_i$ and for any non
negative integer $\mu$, $\mu < m$,
$G$ is not a cograph minimal $(1,\mu)$-polar obstruction.
\end{lemma}

	\begin{proof}
	By considering the different cases in the characterization
	of disconnected cograph minimal $(s,1)$-polar obstructions,
	it is not hard to verify that any connected cograph minimal
	$(1,s)$-polar obstruction $G$ that is $(1, s+1)$-polar contains,
	for any non negative integer $\sigma$ such that $\sigma < s$,
	a proper induced subgraph $G'$ that is both, a cograph
	minimal $(1,\sigma)$-polar obstruction and a
	$(1,\sigma +1)$-polar graph.
	
	Let $\mu$ be a positive integer such that $\mu < m$, and
	let $s_1,\dots,s_t$ be integers such that,
	for $i\in\{1,\dots,t\}$, $0 \le s_i \le k_i$ and 
	$\mu = t-1 + \sum_{i=1}^t s_i$. By the choice of $\mu$,
	$s_i < k_i$ for at least one $i \in \{ 1, \dots, t \}$.
	For each $i\in\{1,\dots,t\}$, if $s_i < k_i$ let $H_i$ be a
	proper induced subgraph of $G_i$ that is both, a cograph
	minimal $(1,s_i)$-polar obstruction and a
	$(1,s_i +1)$-polar graph,
	otherwise let $H_i = G_i$.
	Then, by Lemma \ref{teo:SumsOfObstructions},
	$H = H_1 + \dots + H_t$ is a
	cograph minimal $(1,\mu)$-polar obstruction that is a
	proper induced subgraph of $G$, and therefore $G$ is
	not a cograph minimal $(1,\mu)$-polar obstruction.
	\end{proof}

We conclude the analysis of the disconnected cograph
minimal $(1,k)$-polar obstructions by showing that
the cographs described in Lema \ref{teo:SumsOfObstructions}
are the only ones.

\begin{lemma}\label{lem:Decomposition(1,k)Obstructions}
Let $G$ be a disconnected cograph minimal $(1,k)$-polar
obstruction with components $G_1, \ldots, G_t$.
Then, there exist non negative integers $k_1, \ldots, k_t$
such that for each $i \in \{ 1, \ldots, t \}$, $G_i$ is a
connected cograph minimal $(1, k_i)$-polar obstruction
that is a $(1,k_i+1)$-polar cograph, and
$\sum_{i=1}^t k_i=k-t+1$.
\end{lemma}

	\begin{proof}
	Since $G$ is a cograph minimal $(1,k)$-polar obstruction
	we have that, for each $i \in \{1,\dots,t\}$, the component
	$G_i$ of $G$ is a $(1, k)$-polar graph.	
	For each $i\in \{1,\dots,t\}$ and each $v\in V(G_i)$, let
	$k_v$ be the minimum non negative integer such that
	$G_i -v$ is a $(1, k_v)$-polar graph, and let $k_i$ be the
	maximum of $k_v$ on all the vertices $v$ of $G_i$, that is,
	$k_i = \max\{k_v \colon\ v\in V(G_i)\}$.
	Note that for each $i\in\{1,\dots,t\}$ and any
	$v\in V(G_i)$, $G_i -v$ is a $(1,k_i)$-polar graph.
	
	Moreover, we claim that for each $i\in \{1,\dots,t\}$,
	the graph $G_i$ is not $(1,k_i)$-polar.  Suppose for a
	contradiction that for some $i\in\{1,\dots,t\}$, $G_i$ is a
	$(1,k_i)$-polar graph, we will assume $i=1$ without
	loss of generality.  Let $\{X_1, Y_1\}$ be a
	$(1,k_1)$-polar partition of $G_1$, and let $v \in
	V(G_1)$ such that $G_1-v$ is $(1,k_1)$-polar but it 
	is not $(1,k_1-1)$-polar.  Let $\{A, B\}$
	be a $(1,k)$-polar partition of $G-v$.   For every $i \in
	\{ 1, \dots, t \}$ define $A_i$ and $B_i$ in the following
	way, $A_1 = A \cap V(G_1 -v)$, $B_1 = B \cap
	V(G_1 -v)$, and for each $j\in \{2,\dots,t\}$, let
	$A_j = A\cap V(G_j)$ and $B_j = B\cap V(G_j)$.  Then
	$\{(A\setminus A_1) \cup X_1, (B\setminus B_1)\cup Y_1\}$
	is a $(1,k)$-polar partition of $G$, a contradiction.
	
	Thus, for each $i\in \{1,\dots,t\}$, $G_i$ is a connected
	cograph	minimal $(1,k_i)$-polar obstruction that is
	$(1,k)$-polar, and in consequence $\overline{G_i}$ is a
	disconnected cograph minimal $(k_i,1)$-polar obstruction
	that is a $(k,1)$-polar graph.
	Observe that by Lemmas \ref{lem:EssentialObstructions} to
	\ref{lem:H+K1} this implies that $\overline G_i$ is
	one of $2K_{k_i +1}, K_2+(\overline{K_2}\oplus K_{k_i})$
	or $K_1+(C_4\oplus K_{k_i -1})$, and then, $\overline{G_i}$
	is a disconnected cograph minimal $(k_i,1)$-polar obstruction
	that is $(k_i +1, 1)$-polar.  Equivalently, we have that
	$G_i$ is	a connected cograph minimal
	$(1,k_i)$-polar obstruction that
	is a $(1, k_i +1)$-polar graph.
	
	Finally, by Lemmas \ref{teo:SumsOfObstructions} and
	\ref{cor:SumsOfObstructions} we have that, for
	$m = t - 1 + \sum_{i=1}^t k_i$, $G$ is a cograph minimal
	$(1,m)$-polar obstruction that is a $(1,m+1)$-polar graph,
	and that $G$ is not a cograph minimal $(1,\mu)$-polar
	obstruction for any integer $\mu$ with $0\le \mu < m$.
	Thus, since we are assuming that $G$ is a cograph minimal
	$(1,k)$-polar obstruction,
	we have that $k=m$ and the result follows.

	\end{proof}
		
	Hence, we are ready to state our main result.
	
	\begin{theorem}
	Let $G$ be a cograph, and let $s$ be an integer, $s \ge 2$.
	Then $G$ is a minimal $(s,1)$-polar obstruction if and only
	if it is one of the following:
	\begin{itemize}
		\item One of the four essential obstructions depicted
			in Figure \ref{fig:essentials}, i.e., $K_1 + 2K_2,
			\overline{K_2} + C_4, 2P_3$ or $K_1 +
			(\overline{P_3} \oplus \overline{K_2})$.
		
		\item $2 K_{s+1}$.
		
		\item $K_2 + (\overline{K_2} \oplus K_s)$.
		
		\item $K_1 + (C_4 \oplus K_{s-1})$.
		
		\item $\overline{(s+1) K_2}$.
		
		\item The complement of $G$ is disconnected with
			components $G_1, \dots, G_t$, such that $t \le s$,
			$G_i$ is the complement of a non-essential
			disconnected cograph minimal $(s_i, 1)$-polar
			obstruction and $\sum_{i=1}^t s_i = s-t+1$.
	\end{itemize}
	\end{theorem}
	
	\begin{proof}
	It is an immediate consequence of all previous lemmas.
	\end{proof}

To finish this section, we will prove that the four essential
obstructions in Figure \ref{fig:essentials} constitute the set
of minimal forbidden induced subgraphs for a cograph to
admit an $(s,1)$-polar partition for some integer $s$, $s \ge
2$.


\begin{lemma}\label{lem:OrderOfObstructions}
Let $s$ be an integer. If $G$ is a cograph minimal
$(s,1)$-polar obstruction that is not essential,
then the order of $G$ is at least $s+1$.
\end{lemma}

	\begin{proof}
	
	We will proceed by induction on $s$.
	The unique cograph minimal $(s,1)$-polar obstruction for
	$s=0$ is $2 K_1$, while the unique two cograph minimal
	$(1,1)$-polar obstructions are $C_4$ and $2K_2$.
	This deals with the base case.
	
	Let $s$ be an integer, $s\ge 2$, and suppose that for every
	integer $N$ such that $N<s$, if $H$ is a non-essential
	cograph minimal $(N,1)$-polar obstruction, then $H$ has at
	least $N+1$ vertices.
	
	Let $G$ be a non-essential cograph minimal $(s,1)$-polar
	obstruction. Observe that if $G$ is disconnected, then by
	Lemmas \ref{lem:TwoComponents}, \ref{lem:H+K2} and
	\ref{lem:H+K1}, the order of $G$ is strictly greater than $s$.
	Else, $G$ is a connected cograph and its complement,
	$\overline G$, is a disconnected cograph minimal $(1,s)$-polar
	obstruction; Lemmas \ref{theo:Components Of(1,k)Obstructions}
	and \ref{lem:Decomposition(1,k)Obstructions} imply that either
	$\overline G$ is isomorphic to $(s+1)K_2$, which clearly has
	strictly more than $s+1$ vertices, or the components
	of $\overline{G}$ are $G_1,\dots,G_t$ for some
	integer $t\in \{2,\dots,s\}$, where $G_i$ is a cograph minimal
	$(1,s_i)$-polar obstruction for $1\le i\le t$ and some nonnegative
	integer $s_i$, and $\sum_{i=1}^t s_i = s-t+1$. Nevertheless,
	in the latter case we have by induction hypothesis that for every
	$i\in\{1,\dots,t\}$, the order of $G_i$ is at least $s_i +1$,
	which implies that
	\begin{eqnarray*}
	|\overline G| & = & |G_1| + \dots + |G_t| \\
		& \geq & (s_1 +1) + \dots + (s_t +1) \\
		& = & s_1 + \dots +s_t +t \\
	 	& = & s - t + 1 + t \\
		& = & s+1,
	\end{eqnarray*}
	which ends the proof.
	\end{proof}

\begin{theorem} \label{thm:essentials}
Let $G$ be a cograph. Then $G$ admits an $(s,1)$-polar
partition for some $s \ge 2$ if and only if it does not contain
any of the essential obstructions (Figure \ref{fig:essentials})
as an induced subgraph.
\end{theorem}

	\begin{proof}
	Let $G$ be a cograph such that for every integer $s$, $s\ge 2$,
	$G$ is not an $(s,1)$-polar cograph. Particularly $G$ is not a
	$(n,1)$-polar cograph, where $n$ stands for the order of $G$,
	and therefore $G$ contains a cograph minimal $(n,1)$-polar
	obstruction $H$ as induced subgraph. If $H$ is not essential,
	then, by Lemma \ref{lem:OrderOfObstructions} we have that
	$H$ has order at least $n+1$, which is impossible since $H$
	is a subgraph of $G$. Thus $G$ contains an essential obstruction
	as an induced subgraph. The converse implication follows directly
	from Lemma \ref{lem:EssentialObstructions}.
	\end{proof}

\section{On the number of cograph minimal $(s,1)$-polar
obstructions}\label{sec:Number}

Taking into consideration the number of cograph minimal
$(s,1)$-polar obstructions for $s \in \{ 0, 1, 2 \}$, it would
seem that the number of this obstructions does not grow
too fast.   Nonetheless, a quick estimation shows that
the growth rate of the families of minimal obstructions
is subexponential at best, and we have exponential upper
bounds (with an extremely bad overestimation).

Let $s$ be an integer, $s \geq 2$. In view of Lemmas
\ref{lem:EssentialObstructions} to \ref{lem:H+K1}, there
are exactly seven disconnected cograph minimal $(s,1)$-polar
obstructions, namely $2K_{s+1}$,
$K_1 + (C_4\oplus K_{s-1})$,
$K_2 + (\overline{K_2}\oplus K_s)$, and the four essential
obstructions depicted in Figure \ref{fig:essentials}.
Observe that the complements of the first three
graphs mentioned above are the unique connected cograph
minimal $(1,s)$-polar obstructions that are
$(1, s+1)$-polar cographs.

On the other hand,
to count the number of connected cograph minimal
$(s,1)$-polar obstructions is equivalent to count the number
of disconnected cograph minimal $(1,s)$-polar obstructions.
Furthermore, by Lemma
\ref{lem:Decomposition(1,k)Obstructions},
each disconnected cograph minimal $(1,s)$-polar obstruction
$G$ with components $G_0, \dots, G_k$ satisfies that $G_i$ is
a connected cograph minimal $(1,s_i)$-polar obstruction
that is a $(1, s_i+1)$-polar cograph for each
$i\in\{0,\dots,k\}$, with
$s = s_0 + \dots + s_k + k$ where each term is a non
negative integer. Since there is exactly one connected
cograph minimal $(1, s_i)$-polar obstruction for $s_i
\in \{ 0,1 \}$, and there are exactly three of them
which are connected for $s_i \ge 2$ we have the following.

\begin{proposition}\label{ObstructionsByPartition}
Let $s$ be an integer, $s \ge 2$. If $s$ is expressed as a sum
of non negative integers, $s = s_0 + s_1 + \dots + s_k + k$,
and there are exactly $n$ of the terms $s_i$ greater than 1,
then there are at most $3^n$ non isomorphic
disconnected cograph minimal
$(1,s)$-polar obstructions $G$
with connected components $G_0,\dots,G_k$ such that
$G_i$ is a cograph minimal $(1, s_i)$-polar obstruction
for each $i \in \{ 0, \dots, k \}$.
\end{proposition}

Let $s$ be a non negative integer, and let $D(s)$ be the number
of distinct ways in which $s$ can be expressed as a sum
$s = s_0 + s_1 + \dots + s_k + k$, where $k\ge 1$ and
$s_i$ is a non negative integer for each $i \in \{ 0,1
\dots,k\}$, and where we are
considering two of this representations of $s$ as the same
when they correspond to a permutation of the terms $s_i$.
Thus, the preceding lemma
gives straightforward bounds for the number of disconnected
cograph minimal $(1,s)$-polar cographs in terms of $D(s)$.

\begin{lemma}\label{Bounds}
Let $s$ be an integer, $s \ge 2$. Then the number of
disconnected cograph minimal $(1,s)$-polar obstructions,
$n(s)$, is such that
$$D(s) \le n(s) \le 3^{m} \cdot D(s) < 3^{s/2} \cdot D(s),$$
where $m$ is the maximum possible number of terms
$s_i$ greater that one in a decomposition
$s = s_0 + \dots + s_k +k$ of $s$ with $k \ge 1$.
\end{lemma}

	\begin{proof}
	The left inequality is due to the fact that for each
	decomposition of $s$
	as a sum of non negative integers
	$s = s_0 + \dots + s_k +k$, there is at least one
	disconnected cograph minimal $(1,s)$-polar obstruction.
	The inequality in the middle is an direct consequence
	of Lemma \ref{ObstructionsByPartition}, while the last
	inequality follows from the trivial fact that $m < s/2$.
	\end{proof}

It is evident that every non negative integer $s$ is
decomposed in a sum of non negative integers
$s = s_0 + \dots + s_k +k$ with $k\ge 1$, if and only if
$s+1$ is decomposed in a sum of positive integers
$s+1 = s'_0 + \dots + s'_k$, where $s'_i = s_i +1$ for
$0 \le i \le k$. Thus, the number of distinct ways in which
a positive integer $s$ can be written as a sum of positive
integers, denoted $p(s)$, satisfies the equality
$D(s) = p(s+1)-1$. The parameter $p(s)$ has been
extensively studied, and particularly, Hardy and Ramanujan
gave in 1918 the following asymptotic approximation.

\begin{theorem}
Let $p(n)$ be the number of ways of writing the
positive integer $n$ as a sum of positive integers, where
the order of the terms is not considered. Then
$$p(n) \sim \frac{1}{4n\sqrt{3}} \exp\left(\pi\sqrt{\frac{2n}{3}}\right).$$
\end{theorem}


\section{Conclusions} \label{sec:Conclusions}

Exact lists of cograph minimal $(s,k)$-polar obstructions
are known when $\max \{ s, k \} \le 2$, and also when
$\min \{ s, k \} = 0$.   The results in the present work
seem to indicate that there are too many cograph minimal
$(s,k)$-polar obstructions to expect to find exhaustive
lists for arbitrary values of $s$ and $k$.   Nonetheless,
it was a pleasant surprise to find a recursive characterization
which is rather simple to obtain all the cograph minimal
$(s,1)$-polar obstructions.   This result makes us wonder
whether a similar result may be achieved for any values
of $s$ and $k$.   In particular, we already have some
encouraging partial results for the case when $s = k$,
showing that maybe a combination of recursion together
with a classification of some families of minimal obstructions
may cover the whole family of minimal obstructions.

Also, taking into account the results in \cite{ekimDAM156}
and Theorem \ref{thm:essentials}, it seems possible to find
the complete list of minimal obstructions to the problem of
recognizing $(s, t)$-polar cographs, for some integer $t$
and a fixed integer $s$, $s \ge 2$.

\end{document}